\documentclass[12pt]{amsart}
\usepackage{amstext,amsfonts,amssymb,amscd,amsbsy,amsmath,verbatim,fullpage}
\usepackage[alphabetic,lite,backrefs]{amsrefs} 
\usepackage{ifthen}
\usepackage{color,tikz}
\usepackage{amsthm}
\usepackage{latexsym}
\usepackage[all]{xy}
\usepackage{enumerate}

\newtheorem{lemma}{Lemma}[section]
\newtheorem{theorem}[lemma]{Theorem}

\newtheorem{prop}[lemma]{Proposition}

\newtheorem{cor}[lemma]{Corollary}

\newtheorem{claim*}{Claim}
\newtheorem{thm}[lemma]{Theorem}
\newtheorem{defn}[lemma]{Definition}

\theoremstyle{remark}
\newtheorem{remark}[lemma]{Remark}
\newtheorem{example}[lemma]{Example}

\newcommand{\CC}{\mathbb C}
\newcommand{\PP}{\mathbb P}

\renewcommand{\AA}{\mathbb A}

\newcommand{\GG}{\mathbb G}

\newcommand{\xX}{\mathcal{X}}
\newcommand{\str}{\text{str}}

\title{A transfer principle for unirationality}

\author{Daniel Erman}
\author{Eric Riedl}

\thanks{DE was supported by NSF DMS-2200469 and ER was supported by NSF CAREER grant DMS-1945944.}

\begin{document}

\maketitle

\begin{abstract}
    We apply ideas related to the strength of polynomials to provide new cases of unirational hypersurfaces. It is famously known that hypersurfaces that are smooth in very high codimension are unirational, and a simple corollary then implies that any polynomial of sufficiently high strength will give rise to a unirational hypersurface. Our main result shows that unirationality is preserved under a substitution of high collective strength. In particular, we prove that polynomials of sufficiently high secondary strength are unirational. Along the way, we introduce a ``transfer principle,'' showing that polynomials of high collective strength have Fano schemes defined by polynomials of high collective strength. This gives an alternate proof of a result of Xi Chen on unirationality of Fano schemes, and proves a weakened form of the de Jong-Debarre Conjecture. Combined with some ideas of Starr, this implies a version of Kazhdan and Ziegler's result about the universality of complete intersections of polynomials.
\end{abstract}

\section{Introduction}

Consider a simple unirational hypersurface like $V(z_0z_1+z_2z_3)\subseteq \mathbb P^3$ and polynomials $g_0, g_1,g_2, g_3\in \CC[x_0,\dots, x_n]$ of degree $d$, where $n\gg d$.  The Ananyan-Hochster Principle (see \cite{ess-bulletin}
for a more precise asymptotic articulation of this principle) provides a heuristic for understanding properties of polynomials in many variables; it suggests, for instance,  that if $V(g_0, g_1,g_2, g_3)$ is a complete intersection that is smooth in sufficiently high codimension (relative to $d$), then the polynomials $g_0, g_1,g_2, g_3$ will ``act like'' independent variables $z_0, z_1,z_2, z_3$.   In this note, we ask:  does this heuristic apply to unirationality?  That is to say, will the hypersurface  $V(g_0g_1+g_2g_3)\subseteq \PP^n$  inherit unirationality from $V(z_0z_1+z_2z_3)$?   Our main result implies that this is indeed the case and that unirationality can be ``transferred'' in the manner suggested by this heuristic.

\begin{theorem}
\label{thm:MainFThm}
Let $g_0, \dots, g_n\in \CC[y_0, \dots, y_N]$ be homogeneous polynomials of degree $d$ such that $V(g_0, \dots, g_n)\subseteq \mathbb P^N$ is a complete intersection that is smooth in sufficiently high codimension (relative to $d,n$).
If $F(x_0, \dots, x_n)$ is a homogeneous polynomial defining a unirational hypersurface in $\PP^n$, 
then $F(g_0, \dots, g_n)$ defines a unirational hypersurface in $\PP^N$.

\end{theorem}

Theorem \ref{thm:MainFThm} leads to new families of unirational varieties with very large singular locus.  This includes hypersurfaces defined by polynomials of high secondary strength (see Section \ref{sec:background} for the definitions of strength, smooth strength and secondary strength), providing a positive answer to ~\cite[Question~4.5]{ermanSecondaryStrength}.  We remark that the degree of $F$ does not factor into the bound in Theorem~\ref{thm:MainFThm}; this is a bit unusual for applications of the Ananyan-Hochster Principle, but it will be a consequence of Theorem~\ref{thm:unirationalZ}.  Of course, if $F(x_0, \dots, x_n)$ is a homogeneous polynomial such that $V(F)\subset \PP^n$ is not unirational, then $F(g_0, \dots, g_n)$ is almost never unirational; see Corollary \ref{cor:notUnirat}.  

Along the way, we provide new proofs, as well as minor extensions, of some already-known results. First, we provide a simpler proof (with a slightly worse bound) of a result of \cite{BeheshtiRiedl} about the dimension of the space of planes in hypersurfaces defined by polynomials of high strength. Using this, we prove a variant of a result of Kazhdan-Ziegler \cite{KazhdanZiegler} on the universality of polynomials of high strength. Our approach is entirely distinct from the additive cominatorics techniques of Kazhdan and Ziegler.  We rely on techniques of Starr \cite{Starr}.  While Starr only considers smooth hypersurfaces, many of his methods apply more generally, and we combine this with an original approach (using strength) to the results about $k$-planes.  This is the version of the result that we prove.

\begin{theorem}\label{thm:starrThm}
    Fix degrees $d_1, \dots, d_c\geq 2$ and an integer $n$.  There exists $m$ depending only on the $d_i$ and $n$ such that the following holds:  if $f_1, \dots, f_c$ are any homogeneous polynomials of degrees $d_1, \dots, d_c$ having collective strength at least $m$, then $n$-plane sections of $V(f_1, \dots, f_c)$ dominate the space of complete intersections of degrees $d_1, \dots, d_c$ in $\mathbb{P}^n$.
\end{theorem}

Our technique for understanding the Fano scheme is via a ``Fano Scheme Transfer Principle,'' which shows that the strength of the polynomials defining the Fano scheme of $X$ are at least as large as the strength of the polynomials defining $X$ (see Lemma \ref{lem:str-transfer} for a more precise statement).  This provides a novel perspective and new proofs of results such as~\cite{chen}, about unirationality of Fano schemes.

\begin{thm}[Fano Transfer Principle]
\label{thm:transferPrinciple}
Let $f_1, \dots, f_c$ be equations of degrees $d_1, \dots, d_c$ and let $X=V(f_1,\dots,f_c)\subseteq \mathbb P^n$.  For any $k\geq 1$, the Fano scheme $F_k(X)$ will be defined by equations whose collective strength is at least as large as the collective strength of the $f_i$.
\end{thm}

In particular, properties that depend on strength such as irreducibility, having the expected dimension, being smooth outside of a high codimension set, being unirational, etc., will be inherited by the Fano scheme as long as the strength of the original equations is high enough.

This paper is organized as follows. In \S\ref{sec:background}, we recall the basic definitions of strength and unirationality and prove some elementary properties about them. In \S\ref{sec:exp dim}, we describe the equations defining the Fano scheme and prove the ``Fano Scheme Transfer Principle.'' In \S\ref{sec:universality}, we prove Theorem \ref{thm:starrThm}. In \S\ref{sec:unirationality}, we show that complete intersections defined by high strength polynomials are unirational. Finally, in \S\ref{sec:secondaryStrengthResult} we prove Theorem \ref{thm:MainFThm}.

\textbf{Acknowledgements:} We gratefully acknowledge helpful conversations with Jordan Ellenberg, David Kazhdan, Steven Sam, Andrew Snowden, and Tamar Ziegler. 

\section{Background and definitions}\label{sec:background}
This paper is concerned with relationships between strength and unirationality. Let us recall the definitions of each.  Throughout, we work over $\CC$.  Of course, the results would hold for any algebraically closed field of characteristic $0$, but we will discuss in Remark~\ref{rmk:char0} where we require the characteristic $0$ assumption.

\begin{defn}
    A variety $Z$ is \emph{unirational} if it admits a dominant rational map from $\PP^n$.
\end{defn}

\begin{prop}
\label{prop:uniratDominantMap}
    If $Y$ is unirational and there is a dominant rational map from $Y$ to $Z$, then $Z$ is also unirational.
\end{prop}
\begin{proof}
    Simply compose the rational maps $\PP^n \dashrightarrow Y \dashrightarrow Z$.
\end{proof}

\begin{cor}
\label{cor:notUnirat}
    Let $f\in \CC[x_0, \dots, x_n]$ be a homogeneous polynomial of degree $e$ and let $g_0, \dots, g_n \in \CC[y_0, \dots, y_N]$ be homogeneous polynomials of degree $d$.  Assume that $g_0, \dots, g_n$ define a dominant rational map $\phi: \PP^N \dashrightarrow \PP^n$.
    If $V(F)$ is not unirational, then $V(F(g_0, \dots, g_n))$ is not unirational.
\end{cor}
\begin{proof}
    Note that $V(F(g_0, \dots, g_n))$ is the closure of $\phi^{-1}(V(F))$.  By Proposition \ref{prop:uniratDominantMap}, if $V(F(g_0,\dots ,g_n))$ is unirational, $V(F)$ would be too.
\end{proof}

\begin{remark} \label{rem:dominantRationalMap}
    Observe that $g_0, \dots, g_n$ define a dominant rational map if the fiber of $\phi$ over any given point of $\PP^n$ has the expected dimension. Since these fibers are defined by the vanishing of various linear combinations of the $g_i$, we see that they will define a dominant rational map if $V(g_0, \dots, g_{n})$ is a complete intersection. In particular, this will happen if the $g_i$ have sufficiently large collective strength.
    
\end{remark}

There are several related notions of strength of polynomials. The notion of strength as defined by Ananyan and Hochster~\cite{AnanyanHochster} was independently introduced much earlier by Schmidt~\cite{Schmidt} and is sometimes also referred to as {\em rank} or {\em Schmidt rank}~\cite{KazhdanZiegler}. In this paper, we will exclusively use the related (but slightly different) notion of {\emph smooth strength}, as it is better suited for obtaining some of our explicit bounds.  We recall both definitions for the reader's convenience.

\begin{defn} \label{defn:strength}
    A polynomial $f$ of degree $d$ has \emph{strength} $s$ if it can be decomposed as $f=\sum_{i=0}^s g_ih_i$ with $1\leq \deg(g_i),\deg(h_i)\leq d-1$, and $s$ is minimal with this property. A collection of polynomials $f_1, \dots, f_n$ has {\emph collective strength} $s$ if any nontrivial linear combination of the polynomials has strength at least $s$ and some linear combination has strength exactly $s$.
\end{defn}

In this paper, we use the related notion of smooth strength.

\begin{defn}\label{defn:smoothStrength}
A (not necessarily homogeneous) polynomial $f\in \CC[x_0, \dots, x_n]$ has \emph{smooth strength} $s$ if it is nonzero and the codimension of the singular locus of $V(f)$ in $\AA^{n+1}$ is $s$.  
A collection $f_1, \dots, f_k \in \CC[x_0, \dots, x_n]$ of  has \emph{collective smooth strength} $s$ if any nontrivial linear combination of $f_1, \dots, f_k$ has smooth strength at least $s$.
\end{defn}

\begin{remark}
By convention, if $f=0$ we will say that it has smooth strength $-\infty$.  By extension, any linearly dependent set of polynomials $f_1, \dots, f_c$ will have collective smooth strength $-\infty$. We also take the convention that the empty set has dimension $\infty$, so that a linear form, for example, will have smooth strength $\infty$. 
\end{remark}

\begin{remark}
The strength and smooth strength of a polynomial $f$ are closely related.  For instance, if $f$ has strength at most $s$ then it has smooth strength at most $2s+2$; this is an easy consequence of the product rule (for details, see for example, the case $r=c=1$ of~\cite[Lemma~2.2]{ess-hartshorne}). The bound in the other direction is less obvious, but if $f$ has smooth strength at most $t$ then it turns out that the strength of $f$ is at most $(\deg f)!t -1$~\cite[Proposition $\text{III}_C$]{Schmidt}.

\end{remark}

Given a polynomial with strength $s$, there is a related notion of secondary strength which describes how a polynomial of low strength might be built out of polynomials of high strength.

\begin{defn}
    A polynomial $f$ of strength $s$ has \emph{secondary strength} $N$ if we can write $f=\sum_{i=0}^s g_ih_i$ as in Definition \ref{defn:strength} where the collective strength of $g_0,\dots, g_s,h_0, \dots ,h_s$ is $N$.  
\end{defn}

\begin{remark}\label{rmk:ErmanQuestion}
    The Ananyan-Hochster Principle suggests that if $f$ has sufficiently large secondary strength, then it will behave like a quadric $x_0y_0+\cdots +x_sy_s$ of rank $2s+2$. Theorem \ref{thm:MainFThm} implies that the vanishing locus of a polynomial of high secondary strength must be unirational, even if the polynomial itself is highly singular.  This provides a positive answer to~\cite[Question~4.5]{ermanSecondaryStrength}.
\end{remark}

    Although this paper is concerned with homogeneous polynomials on $\PP^n$, our definition of smooth strength allows for inhomogeneous polynomials.  The reason for this is that our later arguments involving collective smooth strength are mildly simpler if we also impose conditions on the not-necessarily-homogeneous linear combinations of the $f_i$; see for instance the proof of Lemma~\ref{lem:smoothStrengthSingLocus}

\begin{remark}
    Despite the fact that many of our definitions make sense for inhomogeneous polynomials, the results of this paper require us to think primarily about homogeneous polynomials. In particular, the basic principle that high smooth strength implies unirational is simply not true for inhomogeneous polynomials. For instance, let $X = V(zy^2-x^3-xz^2) \subset \PP^2$ be the smooth cubic plane curve. We can take the cone $Y_N$ over $X$ in some large $\PP^N$, and obtain a projective variety that is visibly not unirational (since it has differential forms pulled back from $X$). On the other hand, if we consider the affine patch of $Y_N$ defined by $z=1$, we see that this is smooth regardless of how large $N$ is. Thus, the smooth strength of this affine piece of $Y_N$ can be arbitrarilty large, even as $Y_N$ is never unirational.
\end{remark}

The following elementary lemma will be useful in \S\ref{sec:exp dim} establishing the ``Fano Transfer Principle'' discussed in the introduction.
\begin{lemma}\label{lem:strengthSpecialization}
The collective strength and collective smooth strength of a set of polynomials can only decrease under specialization of the variables.  More precisely:  let $f_1, \dots, f_c \in \CC[u_0, \dots, u_r]$ be polynomials and let $\ell_0, \dots, \ell_r\in \CC[v_0, \dots, v_t]$ be homogeneous linear forms.  Then the collective smooth strength (resp. collective strength) of the $f_i$ is at least the collective smooth strength (resp. collective strength) of the set of $F_i$, where $F_i:=f_i(\ell_0, \dots, \ell_r)$.

\end{lemma}
\begin{proof}
This is a simple observation. We start by considering collective smooth strength. Suppose $\sum_i a_i f_i$ has a singular locus of codimension $s$. Then since the codimension of the singular locus can only decrease under specialization of variables, we see that $\sum_i a_i F_i = \sum_i a_i f_i(\ell_0, \dots, \ell_r)$ will have singular locus of codimension at most $s$, and hence, the $F_i$ will have collective strength at most $s$.

Now consider the situation for strength. Suppose $f$ has strength $s$ and $f=\sum_{i=0}^s g_ih_i$ with $1\leq \deg(g_i),\deg(h_i) < \deg(f)$ for all $i$. Then 
\[
F=f(\ell_0, \dots, \ell_r) = \sum_{i=0}^s g_i(\ell_0, \dots, \ell_r)h_i(\ell_0, \dots, \ell_r)
\]
and thus the strength of $F$ is at most $s$.  The argument for collective strength is essentially identical and we omit it.
\end{proof}

\begin{lemma}
    Let $f_1, \dots, f_c$ be a collection of polynomials of collective smooth strength at least $s$. For some $k \leq c$ suppose we have numbers $\alpha_{i,j}$ for $1 \leq i \leq c$, $1 \leq j \leq k$. Set $g_j = \sum_i \alpha_{i,j} f_i$, and suppose that $g_1 , \dots , g_k$ are linearly independent. Then the collective smooth strength of $g_1, \dots g_k$ is at least $s$.
\end{lemma}
\begin{proof}
    Since linear combinations of the $g_j$ are precisely linear combinations of the $f_i$, the result follows immediately from the definition of collective smooth strength.
\end{proof}

The utility of the notion of collective smooth strength is illustrated by the following lemma.

\begin{lemma} \label{lem:smoothStrengthSingLocus}
Let $f_1, \dots, f_k$ be polynomials in $n$ variables. Then if the collective smooth strength of $f_1, \dots, f_k$ is at least $s$, $V(f_1, \dots, f_k) \subset \mathbb{A}^n$ is smooth of the expected dimension outside of a set of dimension $n-s+k-1$. Going the other way, if $V(f_1, \dots, f_k)$ is a complete intersection that is smooth outside of a set of dimension $n-k-s$, then the collective strength of $f_1, \dots, f_k$ is at least $s$.
\end{lemma}
\begin{proof}
If $V(f_1, \dots, f_k)$ is not smooth of the expected dimension at $p$, then some linear combination $\sum_i \alpha_i f_i$ must be singular at $p$. Each $\sum_i \alpha_i f_i$ is singular in a locus of dimension at most $n-s$, and there are $k-1$ dimensions of choice of $\alpha_i$ (up to scaling), so the dimension of the locus of $V(f_1,\dots ,f_k)$ that is not smooth of the expected dimension is at most $n-s+k-1$. The first part of the result follows.

Now suppose that $V(f_1, \dots, f_k)$ is smooth outside of a set of dimension $n-k-s$. Let $Z$ be the union of all the singular points of $\sum_i \alpha_i f_i$ as $\alpha_1, \dots, \alpha_k$ vary over all nonzero tuples of complex numbers. Then the singular locus of $V(f_1, \dots, f_k)$ is precisely $Z \cap V(f_1, \dots, f_k)$. Thus, the dimension of $Z$ is at most $n-s$. Thus, any single $\sum_i \alpha_i f_i$ has singular locus of codimension at least $s$. The result follows.
\end{proof}

\begin{cor} \label{cor:irrCompleteIntersection}
Let $f_1, \dots, f_c$ be homogeneous polynomials in $n$ variables of collective smooth strength $s$. If $s \geq 2c+1$, then $V(f_1, \dots, f_c)$ will be irreducible and reduced and $f_1, \dots, f_c$ will be a regular sequence.
\end{cor}
\begin{proof}
By Lemma \ref{lem:smoothStrengthSingLocus}, $V(f_1, \dots, f_c)$ will be smooth of the expected dimension outside of a set of dimension $n-s+c-1$. If $n-s+c-1 < n-c$, we see that $V(f_1, \dots, f_c)$ will be of the expected dimension (and thus that $f_1, \dots, f_c$ will be a regular sequence) and generically reduced. Since it is a complete intersection, this means that it will be connected in codimension 1 by Hartshorne's Connected Theorem~\cite{hartshorneConnected} (see, e.g.~\cite[Theorem~18.12]{EisenbudBook}), and hence, it will be irreducible unless the singular locus is codimension 1 in $V(f_1, \dots, f_c)$. Thus, if $n-s+c -1 \leq n-c-2$, we see that $V(f_1, \dots, f_c)$ will be irreducible.
\end{proof}

\begin{remark}
    Corollary \ref{cor:irrCompleteIntersection} illustrates why collective smooth strength is a better notion for this paper than collective strength. For geometric properties like irreducibility, we see that the amount of smooth strength needed is independent of the degrees of the equations, unlike the situation with strength. 
\end{remark}

\section{Expected dimension of Fano scheme of $k$-planes}\label{sec:exp dim}
Given a projective variety $X\subseteq \PP^n$ we will denote the Fano scheme of $k$-planes in $X$ by $F_k(X) \subseteq \GG(k,n)$. Since Harris, Mazur, and Pandharipande's work, unirationality results about a variety have been closely related with results on Fano schemes. In this section, we use the notion of strength to obtain a variant of such a result, akin to the one in \cite{BeheshtiRiedl}. Our proof is similar, but shorter, at the cost of proving a slightly weaker result. The main idea is that nice properties like high strength or high smooth strength are transferred from the defining equations of $X$ to the defining equations of the Fano scheme $F_k(X)$. 

To make this precise, we need to review the equations that cut out the Fano scheme. (See~\cite[Chapter 6.1.1]{3264} for a more detailed review of the defining equations of a Fano scheme.) The Grassmanian $\GG(k,n)$ is naturally a quotient of an open subscheme in the space $\operatorname{Mat}_{k+1,n+1}$ of $(k+1)\times (n+1)$ matrices, and therefore any subvariety of $\GG(k,n)$ lifts to a subvariety in $\operatorname{Mat}_{k+1,n+1}$.  Similarly, an ideal sheaf $\mathcal J$ on $\GG(k,n)$ naturally lifts to a homogeneous ideal $J$ in the coordinate ring $\CC[u_{0,0}, \dots, u_{k,n}]$ of $\operatorname{Mat}_{k+1,n+1}$. Since understanding $V(J)$ is equivalent to understanding $V(\mathcal J)$, it is enough to understand the equations of $J$.

For each $i=0, \dots, n$, we introduce helper variables $s_0, \dots, s_k$ and we write
\[
X_i := s_0u_{0,i} + s_1u_{1,i} + \cdots + s_ku_{k,i}.
\]
Let $f_1, \dots, f_c$ be the polynomials defining $X$ and recall that $\deg(f_\ell)=d_\ell$. Substituting the $X_i$ into $f_j$, we see that $f_j(X_0, \dots, X_n)$ may be written as
\[
f_\ell(X_0, \dots, X_n) = \sum_{|\alpha| = d_j} s^\alpha g_{\ell,\alpha}(\mathbf{u})
\]
where $s^\alpha=s_0^{\alpha_0}s_1^{\alpha_1}\cdots s_k^{\alpha_k}$ is multi-index notation for a monomial of degree $d_\ell$ in the variables $s_0, \dots, s_k$.  The ideal $J$ is then generated by the equations $g_{\ell,\alpha}(\mathbf{u})$ as we vary $\ell$ and $\alpha$.  In total, there will be $\binom{d_\ell+k}{k}$ monomials $s^{\alpha}$ with of degree $d_\ell$.  Thus, fixing $\ell$ and varying $\alpha$ we will obtain $\binom{d_\ell+k}{k}$ equations $g_{\ell, \alpha}$, each of which has degree $d_\ell$ in the $u$-variables; and if we also vary $\ell$ we will obtain a total of $\sum_{\ell=1}^c \binom{d_\ell+k}{k}$ equations for the ideal $J$.

\begin{lemma}\label{lem:str-transfer}
Let $f_1, \dots, f_c$ be homogeneous polynomials on $\PP^n$ of degrees $d_1, \dots, d_c$, and let $X = V(f_1, \dots, f_c) \subseteq \PP^n$.  Let $\mathcal J$ be the defining ideal of the Fano scheme $F_k(X)\subseteq \GG(k,n)$ and $J\subseteq \CC[u_{0,0}, ..., u_{k,n}]$ be the corresponding homogeneous ideal (as defined in the previous paragraph).  Let $m=\sum_{\ell=1}^c \binom{d_\ell+k}{k}$.

There exists a generating set $(g_1, \dots, g_m)$ of $J$ such that:

\begin{enumerate}
    \item The collective strength of $(g_1,\dots, g_m)$ is at least the collective strength of $(f_1, \dots, f_c)$.
    \item  The collective smooth strength of $(g_1,\dots, g_m)$ is at least the collective smooth strength of $(f_1, \dots, f_c)$.
\end{enumerate}
\end{lemma}
Let us consider an example which captures the main idea of the lemma.  
\begin{example}\label{ex:strength-transfer}
    Let $f=x_0^2 + x_1^2 + \cdots + x_n^2$, set $X=V(f)\subseteq \PP^n$, and consider the Fano scheme $F_1(X)$.  With notation as above, we obtain an ideal $J\subseteq \CC[u_{0,0}, \dots, u_{1,n}]$ as follows.  Let $X_i = su_{0,i}+tu_{1,i}$ for $i=0,\dots, n$. Then $f(X_0, \dots, X_n)$ may be written as 
    \[
    g_{1}(\mathbf{u})s^2 + g_{2}(\mathbf{u})st + g_{3}(\mathbf{u})t^2
    \]
    where each $g_{j}\in \CC[u_{0,0}, \dots, u_{1,n}]$ is homogeneous of degree $2$.  The polynomials $g_1, g_2, g_3$ are generators for the ideal $J$.  Specifically, we have:
    \[
    \begin{cases}
        g_1 & = u_{0,0}^2+u_{0,1}^2 + \cdots + u_{0,n}^2\\
        g_{2} & = 2u_{0,0}u_{1,0} + 2u_{0,1}u_{1,1} + \cdots + 2u_{0,n}u_{1,n}\\
        g_3 & = u_{1,0}^2+u_{1,1}^2 + \cdots + u_{1,n}^2
    \end{cases}.
    \]

Now, the key observation is that we can easily specialize any of the $g_i$ to recover the original equation $f$.  For instance, if we set $u_{0,i}=x_i$ then $g_1$ becomes $f$; if we set $u_{0,i}=u_{1,i}=x_i$ then $g_2$ becomes $2f$; and if we set $u_{1,i}=x_i$ then $g_3$ becomes $f$.  Thus, in a loose sense, the equations $g_i$ are at least as complicated as the original equation $f$.  Lemma~\ref{lem:str-transfer} uses the notion of collective strength (or collective smooth strength) to make this precise.
\end{example}

\begin{remark}
Lemma~\ref{lem:str-transfer} provides a simple and unifying explanation for observations, like those of Xi Chen \cite{chen}, which show that nice properties (irreducibility, smoothness, unirationality etc.) get transferred from $X$ to its Fano schemes.  Specifically, imagine that we fix degrees $d_1, \dots, d_c$ and an integer $k$.  By choosing polynomials $f_1, \dots, f_c$ on $\PP^n$ of degrees $d_1, \dots, d_c$ of collective strength $\geq N$, we have that the Fano scheme $F_k(X)$ will be defined by equations $g_1, \dots, g_m$ of collective strength $\geq N$. Since the number of equations $g_j$ will be independent of $n$, we can ensure $F_k(X)$ will inherit any properties implied by ``large collective strength'', as long as $N\gg 0$ (relative to the number of equations $g_j$).  Of course, we could even iterate this process, and see that iterated Fano schemes like $F_\ell(F_k(X))$ will have nice properties, as long as $N\gg 0$.
\end{remark}

\begin{proof}[Proof of Lemma~\ref{lem:str-transfer}]
As above, we write
\[
X_i := s_0u_{0,i} + s_1u_{1,i} + \cdots + s_ku_{k,i}.
\]
Substituting the $X_i$ into $f_j$, we have
\[
f_\ell(X_0, \dots, X_n) = \sum_{|\alpha| = d_j} s^\alpha g_{\ell,\alpha}(\mathbf{u}),
\]
and the ideal $J$ is then generated by the equations $g_{\ell,\alpha}(\mathbf{u})$ as we vary $\ell$ and $\alpha$.

We claim that these polynomials $g_{\ell,\alpha}$ have collective strength which is at least as large as the collective strength of $f_1, \dots, f_c$.  The basic idea is captured in Example~\ref{ex:strength-transfer}; the general argument is similar but with more extensive bookeeping.

Let us consider some linear combination $\sum_{\ell,\alpha} \beta_{\ell,\alpha} g_{\ell,\alpha}$ where $\beta_{\ell,\alpha}$ are constants. (In the strength case, since the degree of $g_{\ell,\alpha}$ is $d_{\ell}$, we can make this linear combination homogeneous by choosing a degree $d$ and only allowing $\beta_{\ell, \alpha}$ to be nonzero if $d_{\ell} = d$.) To show that this linear combination has high smooth strength, we will specialize the coordinates $u_{*,*}$ to a smaller set of variables that make things easier to understand. Since strength can only go down as we specialize the variables by Lemma~\ref{lem:strengthSpecialization}, this will show the result. 

We let $\phi\colon \mathbb C[\mathbf{u}] \to \mathbb C[\mathbf{x}]$ be the ring map which specializes $u_{i,j} \mapsto \lambda_i x_j$ for general constants $\lambda_i$ to be selected later.
We wish to compute $\phi(g_{\ell,\alpha})$ under this specialization. Recall that $g_{\ell,\alpha}$ is the coefficient of $s^\alpha$ in $f_\ell(X_0, \dots, X_n)$. Specializing the variables via $\phi$, we see that 
$\phi(g_{\ell,\alpha})$ is now the coefficient of $s^{\alpha}$ in 
$
\phi(f_\ell(X_0, \dots, X_n))=f_\ell(\phi(X_0), \dots, \phi(X_n)).
$

We observe that
\begin{align*}
    \phi(X_i) &=s_0\phi(u_{0,i}) + s_1\phi(u_{1,i}) + \cdots + s_k\phi(u_{k,i})\\
    &=s_0\lambda_0x_i + s_1\lambda_1x_i + \cdots + s_k\lambda_kx_i\\
    &=(s_0\lambda_0 + s_1\lambda_1 + \cdots + s_k\lambda_k)x_i.
\end{align*}
Since $f_\ell$ is homogeneous of degree $d_\ell$, it follows that
\[
f_{\ell}(\phi(X_0), \dots, \phi(X_n)) = (s_0\lambda_0 + s_1\lambda_1 + \cdots + s_k\lambda_k)^{d_{\ell}}f_\ell(x_0, \dots, x_n).
\]
The coefficient of $s^{\alpha}$ in the above equation, which is the specialization $\phi$ applied $g_{\ell, \alpha}$, will thus simply equal $\phi(g_{\ell,\alpha}) = \lambda^\alpha \binom{d_\ell}{\alpha} f_{\ell}(\mathbf{x})$, where $\binom{d_\ell}{I}$ is the multinomial coefficient.

Now, to return to the main computation, we needed to bound the strength of our linear combination $\sum_{\ell, \alpha} \beta_{\ell,\alpha} g_{\ell,\alpha}$.  Substituting as above we get that
\[
\phi\left(\sum \beta_{\ell,\alpha} g_{\ell,\alpha}\right) = \sum_{\ell,\alpha} \beta_{\ell,\alpha} \phi(g_{\ell,\alpha}) = \sum_{\ell, \alpha} \beta_{\ell, \alpha} \lambda^\alpha \binom{d_\ell}{\alpha} f_{\ell}
=\sum_{\ell} \left( \sum_{\alpha}\beta_{\ell, \alpha} \lambda^\alpha \binom{d_\ell}{I} \right)f_{\ell}.
\]
For general choice of $\lambda = (\lambda_0, \dots, \lambda_k)$, the coefficient $\sum_\alpha \beta_{\ell, \alpha} \lambda^\alpha \binom{d_\ell}{\alpha}$ of $f_{\ell}$ will be nonzero provided that at least one of the $\beta_{\ell, \alpha}$ was nonzero. Thus, we have a nontrivial $\CC$-linear combination of the $f_\ell$. (If our combination of the $g_{\ell,\alpha}$ was homogeneous to begin with, then this combination of the $f_{\ell}$ will also be homogeneous because the degree of the $g_{\ell, \alpha}$ is $d_{\ell}$ for all $\alpha$.)  Since strength can only go down under specialization by Lemma~\ref{lem:strengthSpecialization} we see that the collective strength of the $g_{\ell, \alpha}$ is at least the collective strength of the $f_{\ell}$.  A similar statement holds for collective smooth strength, also by Lemma~\ref{lem:strengthSpecialization}.
\end{proof}

\begin{remark}\label{rmk:char0}
    The proof of the above illustrates why characteristic 0 (or at least large characteristic) is necessary. In low characteristic, the various binomial coefficients appearing in the proof might vanish. For instance, consider the Fermat hypersurface $\sum_i x_i^{p+1}$ in characteristic $p$. The equations generating $J$ for the space of lines are the equations $\sum_{i=0}^n \binom{p+1}{j} u_{0i}^{j} u_{1i}^{p+1-j} = 0$ for $j = 0, \dots, p+1$. For $j = 2, \dots, p-1$, the binomial coefficient  $\binom{p+1}{j}$ vanishes, and so we have many fewer than the expected number of equations, even though the equations that do appear have high collective strength by the proof of Lemma \ref{lem:str-transfer}.
\end{remark}

\begin{proof}[Proof of Theorem \ref{thm:transferPrinciple}]
    This is immediate from Lemma \ref{lem:str-transfer}.
\end{proof}

\begin{theorem}
\label{thm-expDimFanoScheme}
Let $f_1, \dots, f_c$ be homongeneous polynomials on $\PP^n$ of degree $d_1, \dots, d_c$, and let $X = V(f_1, \dots, f_c)\subseteq \PP^n$. Suppose that the collective smooth strength of the $f_i$ is at least $s$. Then the Fano scheme $F_k(X)$ is irreducible of the expected dimension provided that $s \geq 1+2\sum_{i=1}^c \binom{k+d_i}{k}$.
\end{theorem}
\begin{proof}

Let $\{ g_{\ell,\alpha} \}$ be the defining equations for the ideal $J\subseteq \CC[u_{0,0}, \dots, u_{k,n}]$ as defined prior to the statement of Lemma~\ref{lem:str-transfer}.  The Fano scheme $F_k(X)$ will be irreducible and of the expected dimension provided that the ideal $J$ is prime and the $g_{\ell,\alpha}$ form a regular sequence.  The statement then follows Lemma~\ref{lem:str-transfer} and Corollary~\ref{cor:irrCompleteIntersection}.  

\end{proof}

\section{Universality and a proof Theorem~\ref{thm:starrThm}}\label{sec:universality}
Using Theorem \ref{thm-expDimFanoScheme}, we can use a technique of Jason Starr \cite{Starr} to recover a version of a result of Kazhdan and Ziegler about universality of complete intersections of polynomials.  The following is a more precise version of Theorem~\ref{thm:starrThm}.

\begin{theorem}
    Let $k$ be a positive integer and $X$ be the complete intersection of $f_1, \dots, f_c$ in $\PP^n$ of multidegree $(d_1, \dots, d_c)$ and suppose the collective smooth strength of the $f_i$ is at least $1+2\sum_{i=1}^c \binom{k+d_i}{k}$. Then the space of intersections $X \cap \Lambda$ for $\Lambda$ a $k$-plane dominates the space of complete intersections in $\PP^k$ of multidegree $(d_1, \dots, d_c)$
\end{theorem}
\begin{proof}
    Let $G$ be the space of parameterized $k$-planes, i.e., the space of maps $\phi: \PP^k \to \PP^n$ mapping $\PP^k$ isomorphically onto a $k$-plane. Let $S_{d_i-1}$ be the vector space of degree $d_i-1$ polynomials on $\PP^k$. Consider the map $\alpha: G \to \prod_i S_{d_i}$ given by pulling back $f_1, \dots, f_c$ along a map $\phi \in G$. Then $\alpha^{-1}((0,\dots, 0))$ is naturally a bundle over $F_k(X)$, with fibers given by all the different ways to reparameterize a $k$-plane. By Theorem \ref{thm-expDimFanoScheme}, $\alpha^{-1}((0,\dots,0))$ will have dimension equal to $\dim G - \sum_i \binom{d_i+k}{k}$. It follows that $\alpha$ is dominant, as required.
\end{proof}

For the unirationality results, we need a variant of this universality result for residual complete intersections, which we now define. We start by considering the case of hypersurfaces. Let $X = V(f)$ be a degree $d$ hypersurface containing a $k$-plane $\Lambda$. Let $\Phi$ be a $k+1$ plane containing $\Lambda$. Then $X \cap \Phi$ will be defined by the equation $f|_{\Phi}$ and will either be all of $\Phi$ if $\Phi$ lies in $X$, or will be a hypersurface $V(f|_{\Phi})$ of degree $d$ in $\Phi$. In the case where $\Phi$ is not in $X$, since we know that $\Lambda$ lies in $\Phi \cap X$, we can divide $f|_{\Phi}$ by the equation for $\Lambda$ to get a hypersurface $V(g)$ in $\Phi$ of degree $d-1$; this is called the {\bf residual hypersurface} to $\Lambda$ in $\Phi \cap X$.

We work this out in equations. Choose coordinates so that $\Lambda$ is the $k$-plane $[x_0, \dots, x_k] \mapsto [x_0, \dots, x_k, 0, \dots, 0]$. Let $\Phi$ be the $k+1$-plane $\phi: [x_0, \dots, x_k, ta_{k+1}, \dots, ta_n]$. Then $\phi^*f$ will necessarily be divisible by $t$, since $V(f)$ contains $\Lambda$. Thus, our equation $g$ for the residual hypersurface is given by $g = \frac{1}{t} f(x_0, \dots, x_k, ta_{k+1}, \dots, ta_n)$. Observe that $\Phi$ lies in $X$ if and only if $g=0$. Of course, $g$ also depends on our choice of coordinates $[x_0, \dots, x_n, t]$ on $\Phi$.

It will be useful later to be able to understand the coordinates $g|_{\Lambda}$ of the residual hypersurface to $f$ restricted to $\Lambda$. Restricting $g$ to $\Lambda$ is equivalent to setting $t= 0$. We can take the Taylor series expansion of $f$ around $\Lambda$ to get coordinates. That is, consider $f$ as a polynomial in the variables $x_{k+1}, \dots, x_n$, with coefficients in $\CC[x_0,\dots, x_k]$. The linear part of this expansion is given by $$f = \sum_{i=k+1}^n x_i \frac{\partial f}{\partial x_i}(x_0, \dots, x_k, 0, \dots 0) + I_{\Lambda}^2.$$ Elements of $I_{\Lambda}^2$ become divisible by $t^2$ when we plug in $x_i = ta_i$, and so they become 0 when we compute $g|_{\Lambda}$. It follows that \begin{equation} \label{eqn-formulaForResidual}
    g|_{\Lambda} = \sum_{i = k+1}^n a_i \left. \frac{\partial f}{\partial x_i}\right|_{\Lambda}.
\end{equation}

We can apply this more generally to a complete intersection $V(f_1, \dots, f_c)$, getting the residual equation $g_i$ to each hypersurface $f_i$, and if the intersection of $\Phi$ with $V(f_1, \dots, f_c)$ has every component (except for $\Lambda$) of the expected dimension, then we get a residual complete intersection $V(g_1, \dots, g_c)$. In any case, we get a map from the space of parameterized $\Phi$ containing $\Lambda$ to the set of tuples of polynomials $g_1, \dots, g_c$.

Let $F_{\Lambda}(X)$ be the space of $k+1$-planes in $X$ containing $\Lambda$.  We observe that $F_{\Lambda}(X)$ is naturally a subvariety of the $\PP^{n-k-1}$ that parametrizes $k+1$ planes in $\PP^n$ that contain $\Lambda$.

\begin{prop}
\label{prop:dominantInModuli}
Let $k$ be a positive integer, $X$ be a complete intersection in $\PP^n$ of multidegree $(d_1, \dots, d_c)$, and suppose $n \geq k+ 1 + \sum_i \binom{k+d_i-1}{k}$.  Suppose that $\Lambda$ is a $k$-plane in $X$ and that $F_{\Lambda}(X)$ has the expected dimension. Consider the family of schemes of the form $Y_{\Phi} \cap \Lambda$, where $Y_{\Phi}$ is the residual complete intersection to the $k+1$-plane $\Phi$ containing $\Lambda$. Then this family surjects onto the space of complete intersections of multidegree $(d_1-1, \dots, d_c-1)$ in $\Lambda$.
\end{prop}
\begin{proof}
Choose coordinates such that $\Lambda$ is the $k$-plane $x_{k+1} = \dots = x_n = 0$. Let $(f_1, \dots, f_c)$ be equations for $X$, and let $P = \CC^{n-k}$, where we view a point $\mathbf{a} = (a_{k+1}, \dots, a_n)$ of $P$ as corresponding to the parametrized $k+1$-plane $\Phi_{\mathbf{a}}$ given by $[x_0, \dots, x_k, t] \mapsto [x_0, \dots, x_k, a_{k+1}t, \dots, a_n t]$.  Observe that setting $t=0$ yields $\Lambda$.

As in the paragraph preceding Proposition~\ref{prop:dominantInModuli}, the residual complete intersections yield a map sending each $\Phi_{\mathbf{a}}$ to a tuple $(g_1, \dots, g_c)$ where $g_i$ is a degree $d_i-1$ polynomial on $\Phi\cong \PP^{k+1}$.  Because $\Lambda \cong \PP^k$ is a linear subspace of $\Phi$, we can further restrict each $g_i$ to obtain a degree $d_i-1$ polynomial on $\Lambda$.  

If we let $S_{d_i-1}$ be the vector space of degree $d_i-1$ polynomials on $\Lambda = \PP^k$, then we have constructed a map:
\[
\psi: P \to S_{d_1-1} \times \dots \times S_{d_c-1}
\]
given by taking $(a_{k+1}, \dots, a_n)$ to the equations of the residual complete intersection evaluated at $t=0$. Then $\psi^{-1}((0,\dots, 0))$ is simply the affine cone over $F_{\Lambda}(X)$.\footnote{Since the map $\psi$ depended on the choice of a {\em parametrized} $k+1$-plane $\Phi$ contain $\Lambda$, this meant that $P$ was an affine space.  More specifically, $P$ is the affine cone over the projective space of $(k+1)$-planes containing $\Lambda$, and thus the fiber $\psi^{-1}((0,\dots, 0))$ is not the usual Fano variety, but the affine cone over that variety.} We know that $P$ has dimension $n-k$, and the codomain $S_{d_1-1} \times \dots \times S_{d_c-1}$ has dimension $\sum_i \binom{d_i+k-1}{k}$. By hypothesis, $\psi^{-1}(0,\dots,0)$ is the expected dimension $n-k-\sum_i \binom{d_i+k-1}{k}$. Thus, it follows that $\psi$ is surjective.
\end{proof}

\begin{lemma} \label{lem:everyplaneinbiggerplane}
    Suppose $X$ is a complete intersection in $\PP^n$ of type $(d_1, \dots, d_c)$. Let $k$ be an integer with $n - k - 1 \geq \sum_i \binom{k+1+d_i}{k+1}$. Then every $k$-plane in $X$ lies in a $k+1$ plane.
\end{lemma}
\begin{proof}
    Let $\Lambda$ be a $k$-plane in $X = V(f_1, \dots, f_c)$. Choose coordinates so that $\Lambda$ is the plane $[x_0, \dots, x_k, 0, \dots, 0]$. The $k+1$ planes containing $\Lambda$ have the form $[x_0, \dots, x_k, ta_{k+1}, \dots, ta_n]$, and such a plane lies in $X$ precisely when the polynomials $f_i(x_0, \dots, x_k, ta_{k+1}, \dots, ta_n)$ vanish identically. For each $i$, this corresponds to $\binom{k+1+d_i}{k+1}$ polynomials in $a_{k+1}, \dots, a_n$ having to vanish (corresponding to all the different monomials in $x_0, \dots, x_k, t$). Thus, if $n-k > \sum_i \binom{k+1+d_i}{k+1}$, it follows that this cannot be nonempty.
\end{proof}

\begin{cor}\label{cor:dominantResiduals}
    Let $k$ be an integer. Suppose $X$ is a complete intersection in $\PP^n$ of type $(d_1, \dots, d_c)$ having smooth strength at least $s \geq 1+2 \sum_{i=1}^c \binom{d_i + k+1}{k+1}$. Let $\Lambda$ be a general $k$-plane in $X$. Then the space of intersections of the form $Y_{\Phi} \cap \Lambda$, as $Y_{\Phi}$ varies over all residual complete intersections, surjects onto the space of complete intersections of multidegree $(d_1-1, \dots, d_c-1)$ in $\Lambda$.
\end{cor}

\begin{proof}
    This follows by stringing together several results above. First note that $2 \sum_{i=1}^c \binom{d_i + k+1}{k+1} + 1 > k+2+ \sum_i \binom{k+d_i}{k+1}$, since $\sum_i \binom{d_i+k+1}{k+1} + 1 > k+2$ and $\binom{d_i+k+1}{k+1} \geq \binom{k+d_i}{k+1}$. By Theorem \ref{thm-expDimFanoScheme}, we see that $F_k(X)$ and $F_{k+1}(X)$ have the expected dimension, which by Lemma \ref{lem:everyplaneinbiggerplane} means that for a general $\Lambda$, $F_{\Lambda}(X)$ will have the expected dimension. The result then follows from Proposition \ref{prop:dominantInModuli}.
\end{proof}

\begin{remark}
    Jason Starr proves a version of Corollary \ref{cor:dominantResiduals} in \cite[Proposition 1.3]{Starr} for smooth hypersurfaces of degree in $\PP^n$ with a bound of $n \geq \binom{d+k-1}{k} + k$, slightly better than our bound of $2 \binom{d+k}{k}+2$. The proof techniques are different, and we give our version of the result because it highlights the role played by smooth strength in these calculations. 
\end{remark}

\begin{remark} \label{rem:psiExplanation}
In concrete terms, the above proof shows the following.  For any $k$ and any $(d_1, \dots, d_c)$, assume that $f=(f_1, \dots, f_c)$ has collective smooth strength at least $s=2 \sum_{i=1}^c \binom{d_i+k}{k}+2$ and let $\Lambda$ be a general $k$-plane in $X=V(f_1, \dots, f_c)$.  Then the residual map $\psi$ from the proof of Corollary \ref{cor:dominantResiduals}
\[
\psi_{\mathbf{f},\Lambda}: \CC^{n-k} \to \prod_{i=1}^c S_{d_i-1}
\]
given by 
\[
(a_1, \dots, a_k) \mapsto (\sum a_i \frac{\partial f_1}{\partial x_i},
\sum a_i \frac{\partial f_d}{\partial x_i}, \cdots, \sum a_i \frac{\partial f_c}{\partial x_i})
\]
is surjective.
\end{remark}

\begin{cor} \label{cor:expectedDimensionFewerEquations}
    Let $f_1, \dots, f_c$ be polynomials on $\PP^n$ of degree $d$. For some $m \leq c$ suppose we have numbers $\alpha_{i,j}$ for $1 \leq i \leq c$, $1 \leq j \leq m$. Set $g_j = \sum_i \alpha_{i,j} f_i$, and suppose that $g_1 , \dots , g_m$ are linearly independent. Let $X = V(f_1, \dots, f_c)$ and $Y = V(g_1, \dots, g_m)$. Observe that by definition $X \subset Y$. Suppose $\Lambda \in F_k(X)$ with $F_{\Lambda}(X)$ having the expected dimension. Then $F_{\Lambda}(Y)$ also has the expected dimension.
\end{cor}
\begin{proof}
    This follows directly from the description of the equations cutting out $F_{\Lambda}(X)$ and $F_{\Lambda}(Y)$. Choose coordinates so that $\Lambda = V(x_{k+1}, \dots, x_{n})$. Then $F_{\Lambda}(X)$ is cut out by the expansion of the $f_i$'s around the plane $\Lambda$, and $F_{\Lambda}(Y)$ is cut out by the expansion of the $g_i$'s around $\Lambda$. Since the equations of the $g_i$'s are simply linear combinations of the equations of the $f_i$, the fact that $F_{\Lambda}(X)$ is a complete intersection of the expected dimension implies the same result for $F_{\Lambda}(Y)$.
\end{proof}

\section{Unirationality of smooth complete intersections}
\label{sec:unirationality}
The techniques of this section are based on a result of Paranjape and Srinivas \cite{paranjapeSrinivas} and Harris, Mazur and Pandharipande \cite{hmp}, though they date back to Morin~\cite{morin}.

We need a notation for the bound on strength required. For a non-decreasing sequence of integers $d_1, \dots, d_c$, we inductively define a function $U_{\str}(d_1, \dots, d_c)$ as follows. Define $U_{\str}() = 1$ (with the empty tuple of $d_i$'s). Define $U_{\str}(0, \dots, 0, 1,\dots,1, d_1, \dots, d_c) = U_{\str}(d_1, \dots, d_c)$, that is, we are free to omit leading 0's and 1's in the sequence of degrees. Finally, if all the $d_i$ are at least 2, define $$U_{\str}(d_1,\dots, d_c) = 1+ 2 \sum_i \binom{U_{\str} (d_1-1, \dots, d_c-1)+d_i+1}{d_i}.$$ 
The purpose of this bound is the following:  we show that for any complete intersection of $f_1, \dots, f_c$ of degrees $d_1, \dots, d_c$ which has collective smooth strength at least $U_{\str}(d_1, \dots, d_c)$ is unirational.  For the classical inductive argument (e.g. the basic framework of the argument dating back to \cite{paranjapeSrinivas,hmp,morin}) to work, we require not only unirationality of a general residual hypersurface, but we also require this unirational parametrization to hold in families.  The following theorem will inductively prove that $U_{\str}(d_1, \dots, d_c)$ will satisfy this stronger property.

\begin{theorem}
\label{thm-unirationalityInFamilies}
Let $\pi: \PP V \to B$ be a projective bundle of relative dimension $n$ over some base $B$, and for each $i$ let $\mathcal{H}_i$ be a family of degree $i$ hypersurfaces in each $\PP V$. Let $k = U_{\str}(d_1 -1, \dots, d_c-1)$. Let $\Lambda \to B$ be a family of varieties such that a general fiber of $\Lambda \to B$ is a $k$-plane in $\PP(V_b)$. Suppose that in a general fiber of $\pi$, the $\mathcal{H}_i$ have collective smooth strength at least $U_{\str}(d_1, \dots, d_c)$ and $F_{\Lambda}(\cap_i \mathcal H_i)$ has the expected dimension. Then $\cap_i \mathcal{H}_i$ is unirational over $k(B)$.
\end{theorem}

The result we need about families of complete intersections is slightly different from the one stated in \cite{paranjapeSrinivas}, so we include a complete proof even though many of the ideas are the same.

\begin{proof}[Proof of Theorem \ref{thm-unirationalityInFamilies}]
We use induction on $c$ and the $d_i$. The case of $c=1$, $d_1 = 2$ is clear, since quadrics are unirational after choice of a point. If any of the $d_i$ are 1, we can simply view the family as a family of complete intersections in a lower-dimensional projective space. Now suppose we have a family $\xX$ of complete intersections as described in the statement. Restricting to an open set of $B$, we may assume that $\PP V = \PP^n \times B$ and select coordinates so that $\Lambda$ is given by the vanishing of the last $n-k$ coordinates in $\PP^n$. Set $r = U_{\str}(d_1-2, \dots, d_c - 2)$. 

Then for some $b \in B$ and a general $k+1$-plane $\Phi$ in $\PP^n$ that contains $\Lambda$, we get a residual complete intersection $Y_{\Phi,b}$ to $\Lambda$, which has type $(d_1-1, \dots, d_c-1)$. Let $B''$ be the subvariety of $\CC^{n-k} \times \GG(r,\Lambda) \times B$ consisting of tuples $(\Phi, \theta, b)$ where $\Phi$ is a $k+1$ plane containing $\Lambda$ such that $Y_{\Phi}$ is a complete intersection, $\theta$ is an $r$-plane in $Y_{\Phi} \cap \Lambda$ and $b$ is a complete intersection paramterized by $B$. Let $B'$ be the union of the irreducible components of $B''$ that dominate $B$.

We claim that $B'$ is irreducible and rational. To see this, we study for fixed $\mathbf{f} = (f_1, \dots, f_c)$ and $\Lambda$ the map $\psi_{\mathbf{f},\Lambda}\colon \CC^{n-k} \to \prod_{i=1}^c S_{d_i-1,k}$ given by taking residuals as in Remark \ref{rem:psiExplanation}.
As noted in (\ref{eqn-formulaForResidual}) and in Remark \ref{rem:psiExplanation}, 
this map $\psi_{\mathbf{f},\Lambda}$ is given by 
$$ (a_{k+1}, \dots, a_n) \mapsto \left(\sum_{i=k+1}^n a_i \frac{\partial f_1}{\partial x_i}, \dots, \sum_{i=k+1}^n a_i \frac{\partial f_c}{\partial x_i} \right). $$
By Proposition \ref{prop:dominantInModuli} and our hypothesis that a general $F_{\Lambda}(X)$ has the expected dimension, we see that this map will be surjective. 
Since $\psi_{\mathbf{f},\Lambda}$ is a linear map of affine spaces, the fibers of the map will be linear subspaces of $\CC^{n-k}$ of constant dimension. 

Now we consider the incidence variety $I_{\mathcal H} \subseteq \prod_{i=1}^c S_{d_i-1,k} \times \GG(r, \Lambda)\times B$ where $I_{\mathcal H}$ consists of the tuples $(\mathbf{g}, \theta, \mathbf{f})$ where $\theta \subseteq V(\mathbf{g})$; note that $I_{\mathcal H}$ only depends on the initial families of hypersurfaces $\mathcal H_i$, and note that by the discussion above, $I_{\mathcal{H}}$ dominates $\prod_{i=1}^c S_{d_i-1,k}$.  Since $\GG(r,\Lambda)\times B$ is irreducible and rational and since $I_{\mathcal H}$ is a vector bundle over an open set in $\GG(r,\Lambda)\times B$, it follows that $I_{\mathcal H}$ is irreducible and rational.

Now we consider
\[
\xymatrix{
B'\ar[rr]^{\pi}\ar[rd]&& I_{\mathcal H} \ar[ld]\\
&B&} .
\]
Pulled back to $\mathbf{f} \in B$, the horizontal map will become $\psi_{\mathbf{f},\Lambda}$. Thus, we see that $B'$ is a vector bundle over an open set in $I_{\mathcal H}$, and hence is irreducible and rational. Note that because $B'$ dominates $I_{\mathcal{H}}$, for a general element of $B'$, $F_{\theta}(V(\mathbf{g}))$ will have the expected dimension.

Let $\mathcal{Y}$ be the universal family over $B'$, where $\mathcal{Y}$ is the set of $(p,\Phi, \theta, b)$ where $(\Phi,\theta,b) \in B'$ and $p \in Y_{\Phi}$ is an arbitrary point. Then $\mathcal{Y}$ is a family of complete intersections of type $(d_1-1,d_2-1, \dots, d_c-1)$ in a family of $\PP^{k+1}$'s, together with a choice of $r$-planes; because we chose $r=U_{\str}(d_1-2,\dots,d_c-2)$, our induction hypothesis implies that $\mathcal{Y}$ is unirational over $k(B')$. Thus, $\xX$ is unirational over $k(B)$.
\end{proof}

\section{Composition of functions of high collective strength}
\label{sec:secondaryStrengthResult}

Now we are ready to connect unirationality to polynomials of high secondary strength. Recall by Remark \ref{rem:dominantRationalMap} that if $\phi = [g_0, \dots, g_n]$ is a rational map from $\PP^N$ to $\PP^n$ with $V(g_0, \dots, g_n)$ a complete intersection, it follows that $\phi$ is necessarily surjective. Thus in this case, the preimage $\phi^{-1}(Z)$ is well-defined for any subvariety $Z \subset \PP^n$.

\begin{theorem}
    \label{thm:unirationalZ}
    Let $g_0, \dots, g_n\in \CC[y_0, \dots, y_N]$ be degree $d \geq 2$ homogeneous polynomials of collective smooth strength at least $U_{\str}(d, \dots, d)$ (where $d$ appears $n+1$ times). Let $Z \subset \PP^n$ be a unirational variety. Then if $\phi: \PP^N \to \PP^n$ is the rational map defined by the $g_i$, we have that the closure of $\phi^{-1}(Z)$ is unirational.
\end{theorem}

\begin{proof}
We study the map $\phi$. The locus where $\phi$ is undefined is $B = V(g_0, \dots, g_n)$. Given a point $p = [a_0, \dots, a_n] \in \PP^n$, we have the variety $X_p = \overline{\phi^{-1}(p)}$ which is an intersection of some linear combinations of the $g_i$. Because of the condition on the collective smooth strength of the $g_i$, it follows that $V(g_0, \dots, g_n)$ will be a complete intersection. Hence, the same will be true for the $X_p$. Thus, we have a family of complete intersections $X_p$ in $\PP^N$ parameterized by $\PP^n$. Set $k = U_{\str}(d-1, \dots, d-1)$. We would like to find a $k$-plane $\Lambda$ common to all of the $X_p$ such that $F_{\Lambda}(X_p)$ has the expected dimension. The complete intersection $V(g_0, \dots, g_n)$ will be contained in each $X_p$. Since the $g_i$ have high collective strength, it follows from Theorem \ref{thm-expDimFanoScheme} that we can find a $k$-plane $\Lambda$ in $V(g_0, \dots, g_n)$ such that $F_{\Lambda}(V(g_0,\dots,g_n))$ has the expected dimension. By Corollary \ref{cor:expectedDimensionFewerEquations}, it follows that $F_{\Lambda}(X_p)$ will have the expected dimension. Thus, we can find a simultaneous unirational parameterization for the $X_p$, $\alpha: \PP^{N-n} \times \PP^n \to \PP^N$. 

Now let $\psi: \PP^r \to Z$ be a unirational parameterization. Taking the fiber product of $\psi$ with the restriction of $\phi \circ \alpha$ to the inverse image of $Z$, we obtain a rational map from a rational variety to $Z$, as required.
\end{proof}

\begin{proof}[Proof of Theorem \ref{thm:MainFThm}]
    This is a direct application of Theorem \ref{thm:unirationalZ} with $Z = V(F)$.
\end{proof}

\end{document}